\documentclass[11pt,a4paper]{article}

\usepackage{tikz}
\usepackage{amsthm}
\usepackage{amsmath}
\usepackage{amssymb}
\usepackage{mathrsfs}
\usepackage{geometry}
\usepackage{graphicx}
\usepackage{amsfonts}
\usepackage{epstopdf}
\usepackage{enumerate}
\usepackage{subfigure}
\usepackage{hyperref}

\allowdisplaybreaks

\newcommand{\veps}{\varepsilon}

\newcommand{\md}{\mathrm{d}}
\newcommand{\loc}{{\mathrm{loc}}}

\newcommand{\R}{\mathbb{R}}

\newcommand{\rmnum}[1]{\romannumeral #1} 
\newcommand{\Rmnum}[1]{\uppercase\expandafter{\romannumeral#1}} 

\newcommand{\calE}{\mathcal{E}}
\newcommand{\calF}{\mathcal{F}}

\newcommand{\calV}{\mathcal{V}}
\newcommand{\calW}{\mathcal{W}}

\newcommand{\scrE}{\mathscr{E}}

\newcommand{\frakD}{\mathfrak{D}}
\newcommand{\frakE}{\mathfrak{E}}

\newcommand{\frakR}{\mathfrak{R}}

\newcommand{\myset}[1]{\left\{#1\right\}}
\newcommand{\mybar}[1]{\overline{#1}}

\newtheorem{mythm}{Theorem}[section]
\newtheorem{myprop}[mythm]{Proposition}
\newtheorem{mylem}[mythm]{Lemma}
\newtheorem{mycor}[mythm]{Corollary}
\newtheorem{myrmk}[mythm]{Remark}

\newtheorem{mydef}[mythm]{Definition}

\AtEndDocument{\bigskip{\footnotesize%
  \textsc{Fakult\"{a}t f\"{u}r Mathematik, Universit\"{a}t Bielefeld, Postfach 100131, 33501 Bielefeld, Germany.} \par  
  \textit{E-mail address}: \texttt{ymeng@math.uni-bielefeld.de}
}}

\begin{document}

\title{Construction of Local Regular Dirichlet Form on the Sierpi\'nski Gasket using $\Gamma$-Convergence}
\author{Meng Yang}
\date{}

\maketitle

\abstract{We construct a self-similar local regular Dirichlet form on the Sierpi\'nski gasket using $\Gamma$-convergence of stable-like non-local closed forms. As a continuation of a recent paper by Grigor'yan and the author, we give the first \emph{unified} purely analytic construction of local regular Dirichlet forms that works both on the Sierpi\'nski gasket and the Sierpi\'nski carpet.}

\footnote{\textsl{Date}: \today}
\footnote{\textsl{MSC2010}: 28A80}
\footnote{\textsl{Keywords}: Sierpi\'nski gasket, local regular Dirichlet form, $\Gamma$-convergence, stable-like non-local closed form}
\footnote{The author was supported by SFB1283 of the German Research Council (DFG). The author is very grateful to Professor Alexander Grigor'yan for very helpful discussions.}

\section{Introduction}

Recently, Grigor'yan and the author \cite{GY19} gave a purely analytic construction of a local regular Dirichlet form on the Sierpi\'nski carpet (SC). A natural question is to what extent this method can be applied on fractals. The main purpose of this paper is to apply this method to the Sierpi\'nski gasket (SG), a typical representative of p.c.f. (post critically finite) self-similar sets.

The classical analytic construction of local regular Dirichlet form on the SG was given by Kigami which also works on p.c.f. self-similar sets, see \cite{Kig89,Kig93,Kig01}. The most intrinsically essential ingredient in the construction of Kigami is the so-called compatible condition which gives the following two results, one is a monotonicity result which gives the characterization of the local quadratic form, the other is a harmonic extension which gives explicit functions in the domain of the local quadratic form to be dense in certain function spaces.

We get rid of compatible condition \emph{totally} in our construction, but resistance estimates play a central role. On the one hand, resistance estimates imply so-called weak monotonicity results which also give the characterization of the local quadratic form. On the other hand, resistance estimates imply uniform Harnack inequality which is used to construct functions in the domain of the local quadratic form to be dense in certain function spaces.

Since compatible condition does \emph{not} hold on the SC, but resistance estimates hold both on the SG and the SC, our construction can be applied both on the SG and the SC, but the construction of Kigami can not.

On the SG, besides analytic constructions, Barlow, Perkins \cite{BP88} and Kusuoka, Zhou \cite{KZ92} also gave probabilistic constructions using approximation of Markov chains. All these constructions give the same local regular Dirichlet form due to the uniqueness result given by Sabot \cite{Sab97}.

The ultimate purpose of this paper and \cite{GY19} is to provide a new unified method of construction of local regular Dirichlet forms on a wide class of fractals that uses only self-similar property and ideally should be independent of other specific properties, in particular, p.c.f. property.

The idea of our construction is using $\Gamma$-convergence of stable-like non-local closed forms, which is borrowed from the following classical result.
$$\lim_{\beta\uparrow2}(2-\beta)\int_{\R^d}\int_{\R^d}\frac{(u(x)-u(y))^2}{|x-y|^{d+\beta}}\md x\md y=C(d)\int_{\R^d}|\nabla u(x)|^2\md x$$
for all $u\in W^{1,2}(\R^d)$, where $C(d)$ is some positive constant, see \cite[Example 1.4.1]{FOT11}.

This paper is organized as follows. In Section \ref{sec_statement}, we give statement of the main result. In Section \ref{sec_space}, we list some results about Besov spaces. In Section \ref{sec_resist}, we give resistance estimates. In Section \ref{sec_monotone}, we give two weak monotonicity results. In Section \ref{sec_Harnack}, we give uniform Harnack inequality. In Section \ref{sec_good}, we construct one good function. In Section \ref{sec_BM}, we prove Theorem \ref{thm_BM}.

The paper \cite{Kum00} of Kumagai contains partially similar arguments in spirit.

\section{Statement of the Main Result}\label{sec_statement}

Consider the following points in $\R^2$: $p_0=(0,0)$, $p_1=(1,0)$, $p_2=(1/2,\sqrt{3}/2)$. Let $f_i(x)=(x+p_i)/2$, $x\in\R^2$. Then the SG is the unique non-empty compact set $K$ satisfying $K=f_0(K)\cup f_1(K)\cup f_2(K)$. Let $\nu$ be the normalized Hausdorff measure on $K$ of dimension $\alpha=\log3/\log2$. Let
$$V_0=\myset{p_0,p_1,p_2},V_{n+1}=f_0(V_n)\cup f_1(V_n)\cup f_2(V_n)\text{ for all }n\ge0.$$
Then $\myset{V_n}$ is an increasing sequence of finite sets and $K$ is the closure of $\cup_{n=0}^\infty V_n$.

Let $W_0=\myset{\emptyset}$ and 
$$W_n=\myset{w=w_1\ldots w_n:w_i=0,1,2,i=1,\ldots,n}\text{ for all }n\ge1.$$
For all
\begin{align*}
w^{(1)}&=w^{(1)}_1\ldots w^{(1)}_m\in W_m,\\
w^{(2)}&=w^{(2)}_1\ldots w^{(2)}_n\in W_n,
\end{align*}
denote $w^{(1)}w^{(2)}\in W_{m+n}$ by
$$w^{(1)}w^{(2)}=w^{(1)}_1\ldots w^{(1)}_mw^{(2)}_1\ldots w^{(2)}_n.$$
For all $i=0,1,2$, denote
$$i^n=\underbrace{i\ldots i}_{n\ \text{times}}.$$
For all $w=w_1\ldots w_n\in W_n$, let
$$f_w=f_{w_1}\circ\ldots\circ f_{w_n},$$
$$V_w=f_{w}(V_0),K_w=f_{w}(K).$$

For all $n\ge1$, let $\calV_n$ be the graph with vertex set $V_n$ and edge set given by
$$\myset{(p,q):p,q\in V_n,|p-q|=2^{-n}}.$$

For all $n\ge1$, let $\calW_n$ be the graph with vertex set $W_n$ and edge set given by
$$\myset{(w^{(1)},w^{(2)}):w^{(1)},w^{(2)}\in W_n,w^{(1)}\ne w^{(2)},K_{w^{(1)}}\cap K_{w^{(2)}}\ne\emptyset}.$$
Denote $w^{(1)}\sim_n w^{(2)}$ if $(w^{(1)},w^{(2)})$ is an edge.

For all $u\in L^2(K;\nu)$, $n\ge1$, let $P_nu:W_n\to\R$ be given by
$$P_nu(w)=\frac{1}{\nu(K_{w})}\int_{K_w}u(x)\nu(\md x)=\int_K(u\circ f_w)(x)\nu(\md x),w\in W_n.$$

Our main result is as follows.

\begin{mythm}\label{thm_BM}
There exists a self-similar strongly local regular Dirichlet form $(\calE_\loc,\calF_\loc)$ on $L^2(K;\nu)$ satisfying
\begin{align*}
&\calE_\loc(u,u)\asymp\sup_{n\ge1}\left(\frac{5}{3}\right)^n\sum_{w^{(1)}\sim_nw^{(2)}}\left(P_nu(w^{(1)})-P_nu(w^{(2)})\right)^2,\\
&\calF_\loc=\myset{u\in L^2(K;\nu):\sup_{n\ge1}\left(\frac{5}{3}\right)^n\sum_{w^{(1)}\sim_nw^{(2)}}\left(P_nu(w^{(1)})-P_nu(w^{(2)})\right)^2<+\infty}.
\end{align*}
\end{mythm}

\begin{myrmk}
The above result was also obtained by Kusuoka and Zhou \cite[Theorem 7.19, Example 8.4]{KZ92} using approximation of Markov chains. Here, we use $\Gamma$-convergence of stable-like non-local closed forms.
\end{myrmk}

\section{Besov Spaces}\label{sec_space}

In this section, we list some results about Besov spaces.

For all $\beta\in(0,+\infty)$, define
\begin{align*}
\left[u\right]_{B^{2,2}_{\alpha,\beta}(K)}&=\sum_{n=1}^\infty2^{(\alpha+\beta)n}\int\limits_K\int\limits_{B(x,2^{-n})}(u(x)-u(y))^2\nu(\md y)\nu(\md x),\\
\left[u\right]_{B^{2,\infty}_{\alpha,\beta}(K)}&=\sup_{n\ge1}2^{(\alpha+\beta)n}\int\limits_K\int\limits_{B(x,2^{-n})}(u(x)-u(y))^2\nu(\md y)\nu(\md x),
\end{align*}
and 
\begin{align*}
B_{\alpha,\beta}^{2,2}(K)&=\myset{u\in L^2(K;\nu):\left[u\right]_{B^{2,2}_{\alpha,\beta}(K)}<+\infty},\\
B_{\alpha,\beta}^{2,\infty}(K)&=\myset{u\in L^2(K;\nu):\left[u\right]_{B^{2,\infty}_{\alpha,\beta}(K)}<+\infty}.
\end{align*}

$B^{2,2}_{\alpha,\beta}(K)$ and $B^{2,\infty}_{\alpha,\beta}(K)$ have the following equivalent semi-norms.
\begin{mylem}(\cite[Lemma 3.1]{HK06}, \cite[Theorem 1.1, Lemma 2.1]{Yan18}, \cite[Lemma 2.1, Proposition 11.1]{GY19})\label{lem_equiv}
\begin{enumerate}[(1)]
\item For all $\beta\in(0,+\infty)$, for all $u\in L^2(K;\nu)$, we have
\begin{align*}
[u]_{B^{2,2}_{\alpha,\beta}(K)}&\asymp\calE_\beta(u,u)\asymp\frakE_\beta(u,u),\\
[u]_{B^{2,\infty}_{\alpha,\beta}(K)}&\asymp\sup_{n\ge1}2^{(\beta-\alpha)n}\sum_{w^{(1)}\sim_nw^{(2)}}\left(P_nu(w^{(1)})-P_nu(w^{(2)})\right)^2,
\end{align*}
where
\begin{align*}
\calE_\beta(u,u)&=\sum_{n=1}^\infty2^{(\beta-\alpha)n}\sum_{w^{(1)}\sim_nw^{(2)}}\left(P_nu(w^{(1)})-P_nu(w^{(2)})\right)^2,\\
\frakE_\beta(u,u)&=\int_K\int_K\frac{(u(x)-u(y))^2}{|x-y|^{\alpha+\beta}}\nu(\md x)\nu(\md y).
\end{align*}
\item For all $\beta\in(\alpha,+\infty)$, for all $u\in C(K)$, we have
\begin{align*}
[u]_{B^{2,2}_{\alpha,\beta}(K)}&\asymp E_\beta(u,u),\\
[u]_{B^{2,\infty}_{\alpha,\beta}(K)}&\asymp\sup_{n\ge1}2^{(\beta-\alpha)n}\sum_{w\in W_n}\sum_{p,q\in V_w}\left(u(p)-u(q)\right)^2,
\end{align*}
where
$$E_\beta(u,u)=\sum_{n=1}^\infty2^{(\beta-\alpha)n}\sum_{w\in W_n}\sum_{p,q\in V_w}\left(u(p)-u(q)\right)^2.$$
\end{enumerate}
\end{mylem}

$B^{2,2}_{\alpha,\beta}(K)$ and $B^{2,\infty}_{\alpha,\beta}(K)$ can be embedded into H\"older spaces as follows.

\begin{mylem}\label{lem_cont}(\cite[Theorem 4.11 (\rmnum{3})]{GHL03})
For all $u\in L^2(K;\nu)$, let
\begin{align*}
E(u)&=\sum_{n=1}^\infty2^{(\beta-\alpha)n}\sum_{w^{(1)}\sim_nw^{(2)}}\left(P_nu(w^{(1)})-P_nu(w^{(2)})\right)^2,\\
F(u)&=\sup_{n\ge1}2^{(\beta-\alpha)n}\sum_{w^{(1)}\sim_nw^{(2)}}\left(P_nu(w^{(1)})-P_nu(w^{(2)})\right)^2.
\end{align*}
Then for all $\beta\in(\alpha,+\infty)$, there exists some positive constant $c$ such that
\begin{align*}
|u(x)-u(y)|&\le c\sqrt{E(u)}|x-y|^{\frac{\beta-\alpha}{2}},\\
|u(x)-u(y)|&\le c\sqrt{F(u)}|x-y|^{\frac{\beta-\alpha}{2}},
\end{align*}
for $\nu$-almost every $x,y\in K$, for all $u\in L^2(K;\nu)$.
\end{mylem}

\begin{myrmk}
If $u\in L^2(K;\nu)$ satisfies $E(u)<+\infty$ or $F(u)<+\infty$, then $u$ has a continuous version in $C^{\frac{\beta-\alpha}{2}}(K)$.
\end{myrmk}

\section{Resistance Estimates}\label{sec_resist}

In this section, we give resistance estimates.

We need two techniques from electrical network. The first is the well-known $\Delta$-Y transform, see \cite[Lemma 2.1.15]{Kig01}. The second is shorting and cutting technique, see \cite{DS84}.

For all $n\ge1$, let us introduce an energy on $\calV_n$ given by
$$D_n(u,u)=\sum_{w\in W_n}\sum_{p,q\in V_w}(u(p)-u(q))^2,u\in l(V_n),$$
where $l(S)$ is the set of all real-valued functions on a set $S$. For all $p,q\in V_n$, let
\begin{align*}
R_n(p,q)&=\inf\myset{D_n(u,u):u(p)=0,u(q)=1,u\in l(V_n)}^{-1}\\
&=\sup\myset{\frac{\left(u(p)-u(q)\right)^2}{D_n(u,u)}:D_n(u,u)\ne0,u\in l(V_n)}.
\end{align*}
It is obvious that $R_n$ is a metric on $V_n$. Since the terms $(u(p)-u(q))^2$ are added twice in the summation of $D_n(u,u)$, each edge in $\calV_n$ has conductance 2, or equivalently, resistance $1/2$.

For all $n\ge1$, let us introduce an energy on $\calW_n$ given by
$$\frakD_n(u,u)=\sum_{w^{(1)}\sim_nw^{(2)}}\left(u(w^{(1)})-u(w^{(2)})\right)^2,u\in l(W_n).$$
For all $w^{(1)},w^{(2)}\in W_n$, let
\begin{align*}
\frakR_n(w^{(1)},w^{(2)})&=\inf\myset{\frakD_n(u,u):u(w^{(1)})=0,u(w^{(2)})=1,u\in l(W_n)}^{-1}\\
&=\sup\myset{\frac{\left(u(w^{(1)})-u(w^{(2)})\right)^2}{\frakD_n(u,u)}:\frakD_n(u,u)\ne0,u\in l(W_n)}.
\end{align*}
It is obvious that $\frakR_n$ is a metric on $W_n$. Since the terms $(u(w^{(1)})-u(w^{(2)}))^2$ are added twice in the summation of $\frakD_n(u,u)$, each edge in $\calW_n$ has conductance 2, or equivalently, resistance $1/2$.

\begin{myprop}\label{prop_resist}
For all $n\ge1$, we have
$$R_n(p_0,p_1)=R_n(p_1,p_2)=R_n(p_0,p_2)=\frac{1}{3}\left(\frac{5}{3}\right)^n,$$
$$\frakR_n(0^n,1^n)=\frakR_n(1^n,2^n)=\frakR_n(0^n,2^n)=\frac{1}{2}\left(\frac{5}{3}\right)^n-\frac{1}{2}.$$
\end{myprop}
\begin{myrmk}
In our construction of local regular Dirichlet forms, we only need the asymptotic behaviors of resistances as the case on the SC, see \cite[Theorem 5.1]{GY19}. The p.c.f. property of the SG is not essential, but only makes the calculation simple. 
\end{myrmk}

\begin{proof}
The proof is elementary using $\Delta$-Y transform.
\end{proof}

\begin{mycor}\label{cor_resist}
For all $n\ge1$, for all $w\in W_n$, we have
$$\frakR_n(w,0^n),\frakR_n(w,1^n,),\frakR_n(w,2^n)\le\frac{5}{4}\left(\frac{5}{3}\right)^n.$$
\end{mycor}

\begin{proof}
By symmetry, we only need to consider $\frakR_n(w,0^n)$. Letting $w=w_1\ldots w_{n-2}w_{n-1}w_n$, we construct a finite sequence in $W_n$ as follows.
\begin{align*}
w^{(1)}&=w_1\ldots w_{n-2}w_{n-1}w_{n}=w,\\
w^{(2)}&=w_1\ldots w_{n-2}w_{n-1}w_{n-1},\\
w^{(3)}&=w_1\ldots w_{n-2}w_{n-2}w_{n-2},\\
&\ldots\\
w^{(n)}&=w_1\ldots w_1w_1w_1,\\
w^{(n+1)}&=0\ldots000=0^n.
\end{align*}
For all $i=1,\ldots,n-1$, by cutting technique, we have
\begin{align*}
&\frakR_n(w^{(i)},w^{(i+1)})\\
&=\frakR_n(w_1\ldots w_{n-i-1}w_{n-i}w_{n-i+1}\ldots w_{n-i+1},w_1\ldots w_{n-i-1}w_{n-i}w_{n-i}\ldots w_{n-i})\\
&\le\frakR_{i}(w_{n-i+1}\ldots w_{n-i+1},w_{n-i}\ldots w_{n-i})\le \frakR_i(0^i,1^i)=\frac{1}{2}\left(\frac{5}{3}\right)^i-\frac{1}{2}\le\frac{1}{2}\left(\frac{5}{3}\right)^i.
\end{align*}
Since
$$\frakR_n(w^{(n)},w^{(n+1)})=\frakR_n(w_1^n,0^n)\le\frakR_n(0^n,1^n)=\frac{1}{2}\left(\frac{5}{3}\right)^n-\frac{1}{2}\le\frac{1}{2}\left(\frac{5}{3}\right)^n,$$
we have
$$\frakR_n(w,0^n)=\frakR_n(w^{(1)},w^{(n+1)})\le\sum_{i=1}^n\frakR_n(w^{(i)},w^{(i+1)})\le\frac{1}{2}\sum_{i=1}^n\left(\frac{5}{3}\right)^i\le\frac{5}{4}\left(\frac{5}{3}\right)^n.$$
\end{proof}

\section{Weak Monotonicity Results}\label{sec_monotone}

In this section, we give two weak monotonicity results.

For all $n\ge1$, let
$$a_n(u)=\left(\frac{5}{3}\right)^n\sum_{w\in W_n}\sum_{p,q\in V_w}(u(p)-u(q))^2,u\in l(V_n).$$

We have one weak monotonicity result as follows.

\begin{myprop}\label{prop_monotone1}
There exists some positive constant $C$ such that for all $n,m\ge1$, for all $u\in l(V_{n+m})$, we have
$$a_n(u)\le Ca_{n+m}(u).$$
Indeed, we can take $C=2$.
\end{myprop}

\begin{myrmk}
By the construction of Kigami, the sequence $\myset{a_n(u)}_{n\ge1}$ is indeed monotone increasing and the constant $C$ can be taken to be 1. However, we can only obtain $C=2$ using only resistance estimates.
\end{myrmk}

\begin{proof}
For all $w\in W_n$, for all $p,q\in V_w$ with $p\ne q$, by cutting technique and Proposition \ref{prop_resist}, we have
\begin{align*}
(u(p)-u(q))^2&\le R_m(f_w^{-1}(p),f_w^{-1}(q))\sum_{v\in W_m}\sum_{x,y\in V_{wv}}(u(x)-u(y))^2\\
&=\frac{1}{3}\left(\frac{5}{3}\right)^m\sum_{v\in W_m}\sum_{x,y\in V_{wv}}(u(x)-u(y))^2.
\end{align*}
Hence
\begin{align*}
a_n(u)&=\left(\frac{5}{3}\right)^n\sum_{w\in W_n}\sum_{p,q\in V_w,p\ne q}(u(p)-u(q))^2\\
&\le\left(\frac{5}{3}\right)^n\sum_{w\in W_n}\sum_{p,q\in V_w,p\ne q}\left(\frac{1}{3}\left(\frac{5}{3}\right)^m\sum_{v\in W_m}\sum_{x,y\in V_{wv}}(u(x)-u(y))^2\right)\\
&=2\cdot3\cdot\frac{1}{3}\left(\frac{5}{3}\right)^{n+m}\sum_{w\in W_n}\sum_{v\in W_m}\sum_{x,y\in V_{wv}}(u(x)-u(y))^2\\
&=2\left(\frac{5}{3}\right)^{n+m}\sum_{w\in W_{n+m}}\sum_{x,y\in V_w}(u(x)-u(y))^2\\
&=2a_{n+m}(u).
\end{align*}
\end{proof}

For all $n\ge1$, let
$$b_n(u)=\left(\frac{5}{3}\right)^n\sum_{w^{(1)}\sim_nw^{(2)}}\left(P_nu(w^{(1)})-P_nu(w^{(2)})\right)^2,u\in L^2(K;\nu).$$

We have the other weak monotonicity result as follows.

\begin{myprop}\label{prop_monotone2}
There exists some positive constant $C$ such that for all $u\in L^2(K;\nu)$, for all $n,m\ge1$, we have
$$b_n(u)\le Cb_{n+m}(u).$$
Indeed, we can take $C=36$.
\end{myprop}

This result can be reduced as follows.

For all $n\ge1$, let
$$B_n(u)=\left(\frac{5}{3}\right)^n\sum_{w^{(1)}\sim_nw^{(2)}}\left(u(w^{(1)})-u(w^{(2)})\right)^2,u\in l(W_n).$$

For all $n,m\ge1$, let $M_{n,m}:l(W_{n+m})\to l(W_n)$ be a mean value operator given by
$$(M_{n,m}u)(w)=\frac{1}{3^m}\sum_{v\in W_m}u(wv),w\in W_n,u\in l(W_{n+m}).$$

\begin{myprop}\label{prop_monotone2_graph}
There exists some positive constant $C$ such that for all $n,m\ge1$, for all $u\in l(W_{n+m})$, we have
$$B_n(M_{n,m}u)\le CB_{n+m}(u).$$
\end{myprop}

\begin{proof}[Proof of Proposition \ref{prop_monotone2} using Proposition \ref{prop_monotone2_graph}]
For all $u\in L^2(K;\nu)$, note that
$$P_nu=M_{n,m}(P_{n+m}u),$$
hence
\begin{align*}
b_n(u)&=\left(\frac{5}{3}\right)^n\sum_{w^{(1)}\sim_nw^{(2)}}\left(P_nu(w^{(1)})-P_nu(w^{(2)})\right)^2=B_n(P_nu)\\
&=B_n(M_{n,m}(P_{n+m}u))\le CB_{n+m}(P_{n+m}u)\\
&=C\left(\frac{5}{3}\right)^{n+m}\sum_{w^{(1)}\sim_{n+m}w^{(2)}}\left(P_{n+m}u(w^{(1)})-P_{n+m}u(w^{(2)})\right)^2\\
&=Cb_{n+m}(u).
\end{align*}
\end{proof}

\begin{proof}[Proof of Proposition \ref{prop_monotone2_graph}]
Fix $n\ge1$. Assume that $W\subseteq W_n$ is connected, that is, for all $w^{(1)},w^{(2)}\in W$, there exists a finite sequence $\myset{v^{(1)},\ldots,v^{(k)}}\subseteq W$ with $v^{(1)}=w^{(1)},v^{(k)}=w^{(2)}$ and $v^{(i)}\sim_nv^{(i+1)}$ for all $i=1,\ldots,k-1$. Let
$$\frakD_W(u,u)=
\sum_{\mbox{\tiny
$
\begin{subarray}{c}
w^{(1)},w^{(2)}\in W\\
w^{(1)}\sim_nw^{(2)}
\end{subarray}
$}}
(u(w^{(1)})-u(w^{(2)}))^2,u\in l(W).$$
For all $w^{(1)},w^{(2)}\in W$, let
\begin{align*}
\frakR_W(w^{(1)},w^{(2)})&=\inf\myset{\frakD_W(u,u):u(w^{(1)})=0,u(w^{(2)})=1,u\in l(W)}^{-1}\\
&=\sup\myset{\frac{(u(w^{(1)})-u(w^{(2)}))^2}{\frakD_W(u,u)}:\frakD_W(u,u)\ne0,u\in l(W)}.
\end{align*}
It is obvious that $\frakR_W$ is a metric on $W$.

By H\"older inequality, we have
\begin{align*}
B_n(M_{n,m}u)&=\left(\frac{5}{3}\right)^n\sum_{w^{(1)}\sim_nw^{(2)}}\left(\frac{1}{3^m}\sum_{v\in W_m}\left(u(w^{(1)}v)-u(w^{(2)}v)\right)\right)^2\\
&\le\left(\frac{5}{3}\right)^n\sum_{w^{(1)}\sim_nw^{(2)}}\frac{1}{3^m}\sum_{v\in W_m}\left(u(w^{(1)}v)-u(w^{(2)}v)\right)^2.
\end{align*}

Fix $w^{(1)}\sim_nw^{(2)}$. There exist $i,j=0,1,2$ such that $w^{(1)}i^m\sim_{n+m}w^{(2)}j^m$. For all $v\in W_m$, we have
$$\left(u(w^{(1)}v)-u(w^{(2)}v)\right)^2\le\frakR_{w^{(1)}W_m\cup w^{(2)}W_m}(w^{(1)}v,w^{(2)}v)\frakD_{w^{(1)}W_m\cup w^{(2)}W_m}(u,u).$$
By cutting technique and Corollary \ref{cor_resist}, we have
\begin{align*}
&\frakR_{w^{(1)}W_m\cup w^{(2)}W_m}(w^{(1)}v,w^{(2)}v)\\
&\le\frakR_{w^{(1)}W_m\cup w^{(2)}W_m}(w^{(1)}v,w^{(1)}i^m)+\frakR_{w^{(1)}W_m\cup w^{(2)}W_m}(w^{(1)}i^m,w^{(2)}j^m)\\
&+\frakR_{w^{(1)}W_m\cup w^{(2)}W_m}(w^{(2)}j^m,w^{(2)}v)\\
&\le\frakR_m(v,i^m)+\frac{1}{2}+R_m(v,j^m)\\
&\le\frac{5}{2}\left(\frac{5}{3}\right)^m+\frac{1}{2}\le3\left(\frac{5}{3}\right)^m,
\end{align*}
hence
\begin{align*}
&(u(w^{(1)}v)-u(w^{(2)}v))^2\le3\left(\frac{5}{3}\right)^m\frakD_{w^{(1)}W_m\cup w^{(2)}W_m}(u,u)\\
&=3\left(\frac{5}{3}\right)^m\left(\frakD_{w^{(1)}W_m}(u,u)+\frakD_{w^{(2)}W_m}(u,u)+2\left(u(w^{(1)}i^m)-u(w^{(2)}j^m)\right)^2\right),
\end{align*}
hence
\begin{align*}
&\frac{1}{3^m}\sum_{v\in W_m}\left(u(w^{(1)}v)-u(w^{(2)}v)\right)^2\\
&\le3\left(\frac{5}{3}\right)^m\left(\frakD_{w^{(1)}W_m}(u,u)+\frakD_{w^{(2)}W_m}(u,u)+2\left(u(w^{(1)}i^m)-u(w^{(2)}j^m)\right)^2\right).
\end{align*}
In the summation with respect to $w^{(1)}\sim_nw^{(2)}$, the terms $\frakD_{w^{(1)}W_m}(u,u)$ and $\frakD_{w^{(2)}W_m}(u,u)$ are summed at most 6 times, hence
\begin{align*}
&B_n(M_{n,m}u)\le2\cdot6\left(\frac{5}{3}\right)^n3\left(\frac{5}{3}\right)^m\sum_{w^{(1)}\sim_{n+m}w^{(2)}}(u(w^{(1)})-u(w^{(2)}))^2=36B_{n+m}(u).
\end{align*}
\end{proof}

\section{Uniform Harnack Inequality}\label{sec_Harnack}

In this section, we give uniform Harnack inequality as follows.

\begin{myprop}\label{prop_Harnack}
There exist some constants $C\in(0,+\infty)$, $\delta\in(0,1)$ such that for all $n\ge1$, $x\in K$, $r>0$, for all non-negative harmonic function $u$ on $V_n\cap B(x,r)$, we have
$$\max_{V_n\cap B(x,\delta r)}u\le C\min_{V_n\cap B(x,\delta r)}u.$$
\end{myprop}

\begin{myrmk}
The point of the above result is that the constant $C$ is \emph{uniform} in $n$.
\end{myrmk}

The idea is as follows. First, we use resistance estimates on finite graphs $\calV_n$ to obtain resistance estimates on an infinite graph $\calV_\infty$. Second, we obtain Green function estimates on $\calV_\infty$. Third, we obtain elliptic Harnack inequality on $\calV_\infty$. Finally, we transfer elliptic Harnack inequality on $\calV_\infty$ to uniform Harnack inequality on $\calV_n$.

Let $\calV_\infty$ be the graph with vertex set $V_\infty=\cup_{n=0}^\infty2^nV_n$ and edge set given by
$$\myset{(p,q):p,q\in V_\infty,|p-q|=1}.$$

Locally, $\calV_\infty$ is like $\calV_n$. Let the conductances of all edges be $1/2$. Let $d$ be the graph distance, that is, $d(p,q)$ is the minimum of the lengths of all paths connecting $p$ and $q$. It is obvious that
$$d(p,q)\asymp|p-q|\text{ for all }p,q\in V_\infty.$$
By shorting and cutting technique, we obtain resistance estimates on $\calV_\infty$ from $\calV_n$ as follows.
$$R(x,y)\asymp\left(\frac{5}{3}\right)^{\frac{\log d(x,y)}{\log2}}=d(x,y)^{\frac{\log(5/3)}{\log2}}=d(x,y)^\gamma\text{ for all }x,y\in V_\infty,$$
where $\gamma=\log(5/3)/\log2$.

Let $g_B$ be the Green function in a ball $B$. We have Green function estimates as follows.

\begin{myprop}(\cite[Proposition 6.11]{GHL14})\label{prop_Green}
There exist some constants $C\in(0,+\infty),\eta\in(0,1)$ such that for all $z\in V_\infty,r>0$, we have
$$g_{B(z,r)}(x,y)\le Cr^\gamma\text{ for all }x,y\in B(z,r),$$
$$g_{B(z,r)}(z,y)\ge\frac{1}{C}r^\gamma\text{ for all }y\in B(z,\eta r).$$
\end{myprop}

We obtain elliptic Harnack inequality on $\calV_\infty$ as follows.

\begin{myprop}(\cite[Lemma 10.2]{GT01},\cite[Theorem 3.12]{GH14a})\label{prop_Harnack_infinite}
There exist some constants $C\in(0,+\infty)$, $\delta\in(0,1)$ such that for all $z\in V_\infty,r>0$, for all non-negative harmonic function $u$ on $V_\infty\cap B(z,r)$, we have
$$\max_{B(z,\delta r)}u\le C\min_{B(z,\delta r)}u.$$
\end{myprop}

Now we obtain Proposition \ref{prop_Harnack} directly.

\section{One Good Function}\label{sec_good}

In this section, we construct \emph{one} good function with energy property and separation property.

By standard argument, we have H\"older continuity from Harnack inequality as follows.

\begin{myprop}\label{prop_Holder}
For all $0\le\delta_1<\veps_1<\veps_2<\delta_2\le1$, there exist some positive constants $\theta=\theta(\delta_1,\delta_2,\veps_1,\veps_2)$, $C=C(\delta_1,\delta_2,\veps_1,\veps_2)$ such that for all $n\ge1$, for all bounded harmonic function $u$ on $V_n\cap(\delta_1,\delta_2)\times\R$, we have
$$|u(x)-u(y)|\le C|x-y|^\theta\left(\max_{V_n\cap[\delta_1,\delta_2]\times\R}|u|\right)\text{ for all }x,y\in V_n\cap[\veps_1,\veps_2]\times\R.$$
\end{myprop}
\begin{proof}
The proof is similar to \cite[Theorem 3.9]{BB89}.
\end{proof}

For all $n\ge1$, let $u_n\in l(V_n)$ satisfy $u_n(p_0)=0$, $u_n(p_1)=1$ and
$$D_n(u_n,u_n)=\sum_{w\in W_n}\sum_{p,q\in V_w}(u_n(p)-u_n(q))^2=R_n(p_0,p_1)^{-1}.$$
Then $u_n$ is harmonic on $V_n\cap(0,1)\times\R$, $u_n(x,y)=1-u_n(1-x,y)$ for all $(x,y)\in V_n$ and
$$u_n|_{V_n\cap\myset{\frac{1}{2}}\times\R}=\frac{1}{2},u_n|_{V_n\cap[0,\frac{1}{2})\times\R}<\frac{1}{2},u_n|_{V_n\cap(\frac{1}{2},1]\times\R}>\frac{1}{2}.$$
By Arzel\`a-Ascoli theorem, Proposition \ref{prop_Holder} and diagonal argument, there exist some subsequence still denoted by $\myset{u_n}$ and some function $u$ on $K$ with $u(p_0)=0$ and $u(p_1)=1$ such that $u_n$ converges uniformly to $u$ on $K\cap[\veps_1,\veps_2]\times\R$ for all $0<\veps_1<\veps_2<1$. Hence $u$ is continuous on $K\cap(0,1)\times\R$, $u_n(x)\to u(x)$ for all $x\in K$ and $u(x,y)=1-u(1-x,y)$ for all $(x,y)\in K$.

\begin{myprop}\label{prop_u}
The function $u$ given above has the following properties.
\begin{enumerate}[(1)]
\item There exists some positive constant $C$ such that
$$a_n(u)\le C\text{ for all }n\ge1.$$
\item For all $\beta\in(\alpha,\log5/\log2)$, we have
$$E_{\beta}(u,u)<+\infty.$$
Hence $u\in C^{\frac{\beta-\alpha}{2}}(K)$.
\item
$$u|_{K\cap\myset{\frac{1}{2}}\times\R}=\frac{1}{2},u|_{K\cap[0,\frac{1}{2})\times\R}<\frac{1}{2},u|_{K\cap(\frac{1}{2},1]\times\R}>\frac{1}{2}.$$
\end{enumerate}
\end{myprop}

\begin{proof}
(1) By Proposition \ref{prop_resist} and Proposition \ref{prop_monotone1}, we have
\begin{align*}
a_n(u)&=\lim_{m\to+\infty}a_n(u_{n+m})\\
&\le C\lim_{m\to+\infty}a_{n+m}(u_{n+m})\\
&=C\lim_{m\to+\infty}\left(\frac{5}{3}\right)^{n+m}D_{n+m}(u_{n+m},u_{n+m})\\
&=C\lim_{m\to+\infty}\left(\frac{5}{3}\right)^{n+m}R_{n+m}(p_0,p_1)^{-1}\\
&=C\lim_{m\to+\infty}\left(\frac{5}{3}\right)^{n+m}\left(\frac{1}{3}\left(\frac{5}{3}\right)^{n+m}\right)^{-1}\\
&=3C.
\end{align*}

(2) By (1), for all $\beta\in(\alpha,\log5/\log2)$, we have
$$E_\beta(u,u)=\sum_{n=1}^\infty\left(2^{\beta-\alpha}\frac{3}{5}\right)^na_n(u)\le C\sum_{n=1}^\infty\left(2^{\beta-\alpha}\frac{3}{5}\right)^n<+\infty.$$
By Lemma \ref{lem_equiv} and Lemma \ref{lem_cont}, we have $u\in C^{\frac{\beta-\alpha}{2}}(K)$.

(3) It is obvious that
$$u|_{K\cap\myset{\frac{1}{2}}\times\R}=\frac{1}{2},u|_{K\cap[0,\frac{1}{2})\times\R}\le\frac{1}{2},u|_{K\cap(\frac{1}{2},1]\times\R}\ge\frac{1}{2}.$$
By symmetry, we only need to show that
$$u|_{K\cap(\frac{1}{2},1]\times\R}>\frac{1}{2}.$$
Suppose there exists $(x,y)\in K\cap(1/2,1)\times\R$ such that $u(x,y)=1/2$. Since $u_n-\frac{1}{2}$ is a non-negative harmonic function on $V_n\cap(\frac{1}{2},1)\times\R$, by Proposition \ref{prop_Harnack}, for all $1/2<\veps_1<x<\veps_2<1$, there exists some positive constant $C=C(\veps_1,\veps_2)$ such that for all $n\ge1$
$$\max_{V_n\cap[\veps_1,\veps_2]\times\R}\left(u_n-\frac{1}{2}\right)\le C\min_{V_n\cap[\veps_1,\veps_2]\times\R}\left(u_n-\frac{1}{2}\right).$$
Since $u_n$ converges uniformly to $u$ on $K\cap[\veps_1,\veps_2]\times\R$, we have
$$\sup_{K\cap[\veps_1,\veps_2]\times\R}\left(u-\frac{1}{2}\right)\le C\inf_{K\cap[\veps_1,\veps_2]\times\R}\left(u-\frac{1}{2}\right)=0.$$
Hence
$$u-\frac{1}{2}=0\text{ on }K\cap[\veps_1,\veps_2]\times\R\text{ for all }\frac{1}{2}<\veps_1<x<\veps_2<1.$$
Hence
$$u=\frac{1}{2}\text{ on }K\cap(\frac{1}{2},1)\times\R.$$
By continuity, we have
$$u=\frac{1}{2}\text{ on }K\cap[\frac{1}{2},1]\times\R,$$
contradiction!
\end{proof}

\section{Proof of Theorem \ref{thm_BM}}\label{sec_BM}

We list some basic facts about $\Gamma$-convergence. In what follows, $K$ is a locally compact separable metric space and $\nu$ is a Radon measure on $K$ with full support.

We say that $(\calE,\calF)$ is a \emph{closed form on $L^2(K;\nu)$ in the wide sense} if $\calF$ is complete under the inner product $\calE_1$ but $\calF$ is not necessary to be dense in $L^2(K;\nu)$. If $(\calE,\calF)$ is a closed form on $L^2(K;\nu)$ in the wide sense, we extend $\calE$ to be $+\infty$ outside $\calF$, hence the information of $\calF$ is encoded in $\calE$.

\begin{mydef}\label{def_gamma}
Let $\calE^n,\calE$ be closed forms on $L^2(K;\nu)$ in the wide sense. We say that $\calE^n$ is $\Gamma$-convergent to $\calE$ if the following conditions are satisfied.
\begin{enumerate}[(1)]
\item For all $\myset{u_n}\subseteq L^2(K;\nu)$ that converges strongly to $u\in L^2(K;\nu)$, we have
$$\varliminf_{n\to+\infty}\calE^n(u_n,u_n)\ge\calE(u,u).$$
\item For all $u\in L^2(K;\nu)$, there exists a sequence $\myset{u_n}\subseteq L^2(K;\nu)$ converging strongly to $u$ in $L^2(K;\nu)$ such that
$$\varlimsup_{n\to+\infty}\calE^n(u_n,u_n)\le\calE(u,u).$$
\end{enumerate}
\end{mydef}

\begin{myprop}\label{prop_gamma}(\cite[Proposition 6.8, Theorem 8.5, Theorem 11.10, Proposition 12.16]{Dal93})
Let $\myset{(\calE^n,\calF^n)}$ be a sequence of closed forms on $L^2(K;\nu)$ in the wide sense, then there exist some subsequence $\myset{(\calE^{n_k},\calF^{n_k})}$ and some closed form $(\calE,\calF)$ on $L^2(K;\nu)$ in the wide sense such that $\calE^{n_k}$ is $\Gamma$-convergent to $\calE$.
\end{myprop}

In what follows, $K$ is the SG and $\nu$ is the normalized Hausdorff measure on $K$.

For all $\beta\in(\alpha,+\infty)$, by Lemma \ref{lem_equiv} and Lemma \ref{lem_cont}, denote
\begin{align*}
\calF_\beta&=B_{\alpha,\beta}^{2,2}(K)=\myset{u\in L^2(K;\nu):\left[u\right]_{B_{\alpha,\beta}^{2,2}(K)}<+\infty}\\
&=\myset{u\in L^2(K;\nu):\calE_\beta(u,u)<+\infty}=\myset{u\in L^2(K;\nu):\frakE_\beta(u,u)<+\infty}\\
&=\myset{u\in C(K):E_\beta(u,u)<+\infty},
\end{align*}
denote $\scrE_\beta(u,u)=[u]_{B^{2,2}_{\alpha,\beta}(K)}$ and
$$\beta^*:=\frac{\log5}{\log2}.$$

We have non-local regular closed forms and Dirichlet forms as follows.

\begin{myprop}\label{prop_nonlocal}
For all $\beta\in(\alpha,\beta^*)$, $(\calE_\beta,\calF_\beta)$ is a regular closed form on $L^2(K;\nu)$, $(E_\beta,\calF_\beta)$, $(\scrE_\beta,\calF_\beta),(\frakE_\beta,\calF_\beta)$ are regular Dirichlet forms on $L^2(K;\nu)$. For all $\beta\in[\beta^*,+\infty)$, $\calF_\beta$ consists only of constant functions.
\end{myprop}

\begin{myrmk}
$\calE_\beta$ does not have Markovian property, but $E_\beta,\scrE_\beta,\frakE_\beta$ do have Markovian property. An interesting problem in analysis on fractals is for which value $\beta>0$, $(\frakE_\beta,\calF_\beta)$ is a regular Dirichlet form on $L^2(K;\nu)$. The critical exponent
$$\beta_*:=\sup\myset{\beta>0:(\frakE_\beta,\calF_\beta)\text{ is a regular Dirichlet form on }L^2(K;\nu)}$$
is called the \emph{walk dimension of the SG}. A classical approach to determine $\beta_*$ is using heat kernel estimates and subordination technique to have 
$$\beta_*=\beta^*=\frac{\log5}{\log2},$$
see \cite{Pie00}. The following proof provides an alternative approach without using diffusion.
\end{myrmk}

\begin{proof}[Proof of Proposition \ref{prop_nonlocal}]
For all $\beta\in(\alpha,\beta^*)$. We only need to show that $\calF_\beta$ is uniformly dense in $C(K)$. Then $\calF_\beta$ is dense in $L^2(K;\nu)$. Using Fatou's lemma, we have $\calF_\beta$ is complete under $(\calE_\beta)_1$ metric. Moreover, $\calF_\beta\cap C(K)=\calF_\beta$ is trivially $(\calE_\beta)_1$-dense in $\calF_\beta$ and uniformly dense in $C(K)$. Hence $(\calE_\beta,\calF_\beta)$ is a regular closed form on $L^2(K;\nu)$.

It is obvious that $\calF_\beta$ is a sub-algebra of $C(K)$. By Stone-Weierstrass theorem, we only need to show that $\calF_\beta$ separates points.

Let $u$ be the function in Proposition \ref{prop_u}, then $E_\beta(u,u)<+\infty$, hence $u\in\calF_\beta$.

For all distinct $z_1=(x_1,y_1), z_2=(x_2,y_2)\in K$. Replace $z_i$ by $f_w^{-1}(z_i)$ with some $w\in W_n$ and some $n\ge1$, then there exist $i,j\in\myset{0,1,2}$ with $i\ne j$ such that $z_1\in K_i\backslash K_j$, $z_2\in K_j\backslash K_i$. Without loss of generality, we may assume that $i=0$, $j=1$, then $x_1<1/2<x_2$. By Proposition \ref{prop_u}, we have
$$u(z_1)<\frac{1}{2}<u(z_2).$$
Hence $\calF_\beta$ separates points.

Since $E_\beta,\scrE_\beta,\frakE_\beta$ do have Markovian property, $(E_\beta,\calF_\beta),(\scrE_\beta,\calF_\beta),(\frakE_\beta,\calF_\beta)$ are regular Dirichlet forms on $L^2(K;\nu)$.

For all $\beta\in[\beta^*,+\infty)$. Suppose that $u\in\calF_\beta$ is not constant, then there exists $N\ge1$ such that $b_N(u)>0$. By Proposition \ref{prop_monotone2}, we have
\begin{align*}
\calE_\beta(u,u)&=\sum_{n=1}^\infty2^{(\beta-\alpha)n}\left(\frac{3}{5}\right)^nb_n(u)\ge\sum_{n=N+1}^\infty2^{(\beta-\alpha)n}\left(\frac{3}{5}\right)^nb_n(u)\\
&\ge\frac{1}{C}\sum_{n=N+1}^\infty2^{(\beta-\alpha)n}\left(\frac{3}{5}\right)^nb_N(u)=+\infty,
\end{align*}
contradiction! Hence $\calF_\beta$ consists only of constant functions.
\end{proof}

We need an elementary result as follows.

\begin{mylem}\label{lem_ele}
Let $\myset{x_n}$ be a sequence of non-negative real numbers.
\begin{enumerate}[(1)]
\item $$\varliminf_{n\to+\infty}x_n\le\varliminf_{\lambda\uparrow1}(1-\lambda)\sum_{n=1}^\infty\lambda^nx_n\le\varlimsup_{\lambda\uparrow1}(1-\lambda)\sum_{n=1}^\infty\lambda^nx_n\le\varlimsup_{n\to+\infty}x_n\le\sup_{n\ge1}x_n.$$
\item If there exists some positive constant $C$ such that
$$x_n\le Cx_{n+m}\text{ for all }n,m\ge1,$$
then
$$\sup_{n\ge1}x_n\le C\varliminf_{n\to+\infty}x_n.$$
\end{enumerate}
\end{mylem}
\begin{proof}
The proof is elementary using $\veps$-$N$ argument.
\end{proof}

Take $\myset{\beta_n}\subseteq(\alpha,\beta^*)$ with $\beta_n\uparrow\beta^*$. By Proposition \ref{prop_gamma}, there exist some subsequence still denoted by $\myset{\beta_n}$ and some closed form $(\calE,\calF)$ on $L^2(K;\nu)$ in the wide sense such that $(\beta^*-\beta_n)\calE_{\beta_n}$ is $\Gamma$-convergent to $\calE$. Without loss of generality, we may assume that
$$0<\beta^*-\beta_n<\frac{1}{n+1}\text{ for all }n\ge1.$$

We have the characterization of $(\calE,\calF)$ on $L^2(K;\nu)$ as follows.

\begin{myprop}\label{prop_E}
\begin{align*}
\calE(u,u)&\asymp\sup_{n\ge1}a_n(u)=\sup_{n\ge1}\left(\frac{5}{3}\right)^n\sum_{w\in W_n}\sum_{p,q\in V_w}(u(p)-u(q))^2\\
&\asymp\sup_{n\ge1}b_n(u)=\sup_{n\ge1}\left(\frac{5}{3}\right)^n\sum_{w^{(1)}\sim_nw^{(2)}}\left(P_nu(w^{(1)})-P_nu(w^{(2)})\right)^2,
\end{align*}
$$\calF=\myset{u\in C(K):\sup_{n\ge1}a_n(u)<+\infty}=\myset{u\in L^2(K;\nu):\sup_{n\ge1}b_n(u)<+\infty}.$$
Moreover, $(\calE,\calF)$ is a regular closed form on $L^2(K;\nu)$.
\end{myprop}

\begin{proof}
Recall that
\begin{align*}
E_\beta(u,u)&=\sum_{n=1}^\infty2^{(\beta-\beta^*)n}a_n(u),\\
\calE_{\beta}(u,u)&=\sum_{n=1}^\infty2^{(\beta-\beta^*)n}b_n(u).
\end{align*}

On the one hand, for all $u\in L^2(K;\nu)$, we have
\begin{align*}
\calE(u,u)&\le\varliminf_{n\to+\infty}(\beta^*-\beta_n)\calE_{\beta_n}(u,u)=\varliminf_{n\to+\infty}(\beta^*-\beta_n)\sum_{k=1}^\infty2^{(\beta_n-\beta^*)k}b_k(u)\\
&=\varliminf_{n\to+\infty}\frac{\beta^*-\beta_n}{1-2^{\beta_n-\beta^*}}(1-2^{\beta_n-\beta^*})\sum_{k=1}^\infty2^{(\beta_n-\beta^*)k}b_k(u)\\
&=\frac{1}{\log2}\varliminf_{n\to+\infty}(1-2^{\beta_n-\beta^*})\sum_{k=1}^\infty2^{(\beta_n-\beta^*)k}b_k(u)\le\frac{1}{\log2}\sup_{k\ge1}b_k(u).
\end{align*}
On the other hand, for all $u\in L^2(K;\nu)$, there exists $\myset{u_n}\subseteq L^2(K;\nu)$ converging strongly to $u$ in $L^2(K;\nu)$ such that
\begin{align*}
\calE(u,u)&\ge\varlimsup_{n\to+\infty}(\beta^*-\beta_n)\calE_{\beta_n}(u_n,u_n)=\varlimsup_{n\to+\infty}(\beta^*-\beta_n)\sum_{k=1}^\infty2^{(\beta_n-\beta^*)k}b_k(u_n)\\
&\ge\varlimsup_{n\to+\infty}(\beta^*-\beta_n)\sum_{k=n+1}^\infty2^{(\beta_n-\beta^*)k}b_k(u_n)\\
&\ge\frac{1}{C}\varlimsup_{n\to+\infty}(\beta^*-\beta_n)\sum_{k=n+1}^\infty2^{(\beta_n-\beta^*)k}b_n(u_n)\\
&=\frac{1}{C}\varlimsup_{n\to+\infty}\left[(\beta^*-\beta_n)\frac{2^{(\beta_n-\beta^*)(n+1)}}{1-2^{\beta_n-\beta^*}}b_n(u_n)\right].
\end{align*}
Since $0<\beta^*-\beta_n<1/(n+1)$, we have $2^{(\beta_n-\beta^*)(n+1)}>1/2$. Since
$$\lim_{n\to+\infty}\frac{\beta^*-\beta_n}{1-2^{\beta_n-\beta^*}}=\frac{1}{\log2},$$
we have
$$\calE(u,u)\ge\frac{1}{2(\log2)C}\varlimsup_{n\to+\infty}b_n(u_n).$$
Since $u_n\to u$ in $L^2(K;\nu)$, for all $k\ge1$, we have
$$b_k(u)=\lim_{n\to+\infty} b_k(u_n)=\lim_{k\le n\to+\infty} b_k(u_n)\le C\varliminf_{n\to+\infty} b_n(u_n).$$
Taking supremum with respect to $k\ge1$, we have
$$\sup_{k\ge1}b_k(u)\le C\varliminf_{n\to+\infty}b_n(u_n)\le C\varlimsup_{n\to+\infty}b_n(u_n)\le 2(\log2)C^2\calE(u,u).$$
Hence
$$\frac{1}{2(\log2)C^2}\sup_{k\ge1}b_k(u)\le\calE(u,u)\le\frac{1}{\log2}\sup_{k\ge1}b_k(u).$$

By Lemma \ref{lem_equiv} and Lemma \ref{lem_cont}, we have $\calF\subseteq C(K)$ and
\begin{align*}
&\calE(u,u)\asymp\sup_{k\ge1}a_k(u),\\
&\calF=\myset{u\in C(K):\sup_{k\ge1}a_k(u)<+\infty}.
\end{align*}
Similar to the proof of Proposition \ref{prop_nonlocal}, we have $\calF$ is uniformly dense in $C(K)$, hence $(\calE,\calF)$ is a regular closed form on $L^2(K;\nu)$.
\end{proof}

Now we prove Theorem \ref{thm_BM} as follows.

\begin{proof}[Proof of Theorem \ref{thm_BM}]
For all $n\ge1$, for all $u\in l(V_{n+1})$, we have
\begin{align*}
\frac{5}{3}\sum_{i=0}^2a_n(u\circ f_i)&=\frac{5}{3}\sum_{i=0}^2\left(\frac{5}{3}\right)^n\sum_{w\in W_n}\sum_{p,q\in V_w}(u\circ f_i(p)-u\circ f_i(q))^2\\
&=\left(\frac{5}{3}\right)^{n+1}\sum_{w\in W_{n+1}}\sum_{p,q\in V_w}(u(p)-u(q))^2\\
&=a_{n+1}(u).
\end{align*}
Hence for all $n,m\ge1$, for all $u\in l(V_{n+m})$, we have
$$\left(\frac{5}{3}\right)^m\sum_{w\in W_m}a_n(u\circ f_w)=a_{n+m}(u).$$
For all $u\in\calF$, $n\ge1$, $w\in W_n$, we have
$$\sup_{k\ge1}a_k(u\circ f_w)\le\sup_{k\ge1}\sum_{w\in W_n}a_k(u\circ f_w)=\left(\frac{3}{5}\right)^n\sup_{k\ge1}a_{n+k}(u)\le\left(\frac{3}{5}\right)^n\sup_{k\ge1}a_{k}(u)<+\infty,$$
hence $u\circ f_w\in\calF$.

Let
$$\mybar{\calE}^{(n)}(u,u)=\left(\frac{5}{3}\right)^n\sum_{w\in W_n}\calE(u\circ f_w,u\circ f_w),u\in\calF,n\ge1.$$
Then
\begin{align*}
\mybar{\calE}^{(n)}(u,u)&\ge C\left(\frac{5}{3}\right)^n\sum_{w\in W_n}\varlimsup_{k\to+\infty}a_k(u\circ f_w)\ge C\left(\frac{5}{3}\right)^n\varlimsup_{k\to+\infty}\sum_{w\in W_n}a_k(u\circ f_w)\\
&=C\varlimsup_{k\to+\infty}a_{n+k}(u)\ge C\sup_{k\ge1}a_k(u).
\end{align*}
Similarly
\begin{align*}
\mybar{\calE}^{(n)}(u,u)&\le C\left(\frac{5}{3}\right)^n\sum_{w\in W_n}\varliminf_{k\to+\infty}a_k(u\circ f_w)\le C\left(\frac{5}{3}\right)^n\varliminf_{k\to+\infty}\sum_{w\in W_n}a_k(u\circ f_w)\\
&=C\varliminf_{k\to+\infty}a_{n+k}(u)\le C\sup_{k\ge1}a_k(u).
\end{align*}
Hence
$$\mybar{\calE}^{(n)}(u,u)\asymp\sup_{k\ge1}a_k(u)\text{ for all }u\in\calF,n\ge1.$$
Moreover, for all $u\in\calF$, $n\ge1$, we have
\begin{align*}
\mybar{\calE}^{(n+1)}(u,u)&=\left(\frac{5}{3}\right)^{n+1}\sum_{w\in W_{n+1}}\calE(u\circ f_w,u\circ f_w)\\
&=\left(\frac{5}{3}\right)^{n+1}\sum_{i=0}^2\sum_{w\in W_n}\calE(u\circ f_i\circ f_w,u\circ f_i\circ f_w)\\
&=\frac{5}{3}\sum_{i=0}^2\left(\left(\frac{5}{3}\right)^{n}\sum_{w\in W_n}\calE((u\circ f_i)\circ f_w,(u\circ f_i)\circ f_w)\right)\\
&=\frac{5}{3}\sum_{i=0}^2\mybar{\calE}^{(n)}(u\circ f_i,u\circ f_i).
\end{align*}
Let
$$\tilde{\calE}^{(n)}(u,u)=\frac{1}{n}\sum_{l=1}^n\mybar{\calE}^{(l)}(u,u),u\in\calF,n\ge1.$$
It is obvious that
$$\tilde{\calE}^{(n)}(u,u)\asymp\sup_{k\ge1}a_k(u)\text{ for all }u\in\calF,n\ge1.$$

Since $(\calE,\calF)$ is a regular closed form on $L^2(K;\nu)$, by \cite[Definition 1.3.8, Remark 1.3.9, Definition 1.3.10, Remark 1.3.11]{CF12}, we have $(\calF,\calE_1)$ is a separable Hilbert space. Let $\myset{u_i}_{i\ge1}$ be a dense subset of $(\calF,\calE_1)$. For all $i\ge1$, $\myset{\tilde{\calE}^{(n)}(u_i,u_i)}_{n\ge1}$ is a bounded sequence. By diagonal argument, there exists a subsequence $\myset{n_k}_{k\ge1}$ such that $\myset{\tilde{\calE}^{(n_k)}(u_i,u_i)}_{k\ge1}$ converges for all $i\ge1$. Since
$$\tilde{\calE}^{(n)}(u,u)\asymp\sup_{k\ge1}a_k(u)\asymp\calE(u,u)\text{ for all }u\in\calF,n\ge1,$$
we have $\myset{\tilde{\calE}^{(n_k)}(u,u)}_{k\ge1}$ converges for all $u\in\calF$. Let
$$\calE_{\loc}(u,u)=\lim_{k\to+\infty}\tilde{\calE}^{(n_k)}(u,u)\text{ for all }u\in\calF_{\loc}:=\calF.$$
Then
$$\calE_\loc(u,u)\asymp\sup_{k\ge1}a_k(u)\asymp\calE(u,u)\text{ for all }u\in\calF_\loc=\calF.$$
Hence $(\calE_\loc,\calF_\loc)$ is a regular closed form on $L^2(K;\nu)$. Since $1\in\calF_\loc$ and $\calE_\loc(1,1)=0$, by \cite[Lemma 1.6.5, Theorem 1.6.3]{FOT11}, we have $(\calE_\loc,\calF_\loc)$ on $L^2(K;\nu)$ is conservative.

For all $u\in\calF_\loc=\calF$, we have $u\circ f_i\in\calF=\calF_\loc$ for all $i=0,1,2$ and
\begin{align*}
&\frac{5}{3}\sum_{i=0}^2\calE_\loc(u\circ f_i,u\circ f_i)=\frac{5}{3}\sum_{i=0}^2\lim_{k\to+\infty}\tilde{\calE}^{(n_k)}(u\circ f_i,u\circ f_i)\\
&=\lim_{k\to+\infty}\frac{1}{n_k}\sum_{l=1}^{n_k}\left[\frac{5}{3}\sum_{i=0}^2\mybar{\calE}^{(l)}(u\circ f_i,u\circ f_i)\right]=\lim_{k\to+\infty}\frac{1}{n_k}\sum_{l=1}^{n_k}\mybar{\calE}^{(l+1)}(u,u)\\
&=\lim_{k\to+\infty}\left[\frac{1}{n_k}\sum_{l=1}^{n_k}\mybar{\calE}^{(l)}(u,u)+\frac{1}{n_k}\mybar{\calE}^{(n_k+1)}(u,u)-\frac{1}{n_k}\mybar{\calE}^{(1)}(u,u)\right]\\
&=\lim_{k\to+\infty}\tilde{\calE}^{(n_k)}(u,u)=\calE_\loc(u,u).
\end{align*}
Hence $(\calE_\loc,\calF_\loc)$ on $L^2(K;\nu)$ is self-similar.

For all $u,v\in\calF_\loc$ satisfying $\mathrm{supp}(u),\mathrm{supp}(v)$ are compact and $v$ is constant in an open neighborhood $U$ of $\mathrm{supp}(u)$, we have $K\backslash U$ is compact and $\mathrm{supp}(u)\cap(K\backslash U)=\emptyset$, hence $\delta=\mathrm{dist}(\mathrm{supp}(u),K\backslash U)>0$. Taking sufficiently large $n\ge1$ such that $2^{1-n}<\delta$, by self-similarity, we have
$$\calE_\loc(u,v)=\left(\frac{5}{3}\right)^n\sum_{w\in W_n}\calE_\loc(u\circ f_w,v\circ f_w).$$
For all $w\in W_n$, we have $u\circ f_w=0$ or $v\circ f_w$ is constant, hence $\calE_\loc(u\circ f_w,v\circ f_w)=0$, hence $\calE_\loc(u,v)=0$, that is, $(\calE_\loc,\calF_\loc)$ on $L^2(K;\nu)$ is strongly local.

For all $u\in\calF_\loc$, it is obvious that $u^+,u^-,1-u,\mybar{u}=(0\vee u)\wedge 1\in\calF_\loc$ and
$$\calE_\loc(u,u)=\calE_\loc(1-u,1-u).$$
Since $u^+u^-=0$ and $(\calE_\loc,\calF_\loc)$ on $L^2(K;\nu)$ is strongly local, we have $\calE_\loc(u^+,u^-)=0$. Hence
\begin{align*}
\calE_\loc(u,u)&=\calE_\loc(u^+-u^-,u^+-u^-)=\calE_\loc(u^+,u^+)+\calE_\loc(u^-,u^-)-2\calE_\loc(u^+,u^-)\\
&=\calE_\loc(u^+,u^+)+\calE_\loc(u^-,u^-)\ge\calE_\loc(u^+,u^+)=\calE_\loc(1-u^+,1-u^+)\\
&\ge\calE_\loc((1-u^+)^+,(1-u^+)^+)=\calE_\loc(1-(1-u^+)^+,1-(1-u^+)^+)=\calE_\loc(\mybar{u},\mybar{u}),
\end{align*}
that is, $(\calE_\loc,\calF_\loc)$ on $L^2(K;\nu)$ is Markovian. Hence $(\calE_\loc,\calF_\loc)$ is a self-similar strongly local regular Dirichlet form on $L^2(K;\nu)$.
\end{proof}

\begin{myrmk}
The idea of the construction of $\mybar{\calE}^{(n)},\tilde{\calE}^{(n)}$ is from \cite[Section 6]{KZ92}. The proof of Markovian property is from the proof of \cite[Theorem 2.1]{BBKT10}.
\end{myrmk}

\bibliographystyle{plain}

\def\cprime{$'$}

\end{document}